\documentclass{amsart}

\usepackage{amssymb}
\usepackage{amsmath}

\newtheorem{theorem}{Theorem}[section]
\newtheorem{proposition}[theorem]{Proposition}
\newtheorem{corollary}[theorem]{Corollary}
\newtheorem{lemma}[theorem]{Lemma}

\theoremstyle{definition}

\theoremstyle{remark}
\newtheorem{remark}[theorem]{Remark}

\numberwithin{equation}{section}

\def\P{\ensuremath{\mathsf{p}}}
\def\Q{\ensuremath{\mathsf{q}}}

\newcommand{\Sh}{\ensuremath{\textsf{S}}}

\newcommand{\F}{\ensuremath{\mathcal{F}}}

\newcommand{\dank}{\textsf{Acknowledgments.\ }}

\newcommand{\RR}{\mathbb R}

\newcommand{\metric}{\ensuremath{ \mathrm{g} }}

\def\ol{\overline}

\def\ra{\rightarrow}

\def\In{\subseteq}
\def\tr{\textrm{tr}}
\def\RR{\mathbb{R}}
\def\mc{\mathcal}

\def\fol{\mc{F}}

\begin{document}


\title{Isometries between leaf spaces}

\author{Marcos M. Alexandrino}

\address{ 
Universidade de S\~{a}o Paulo\\
Instituto de Matem\'{a}tica e Estat\'{\i}stica\\
 Rua do Mat\~{a}o 1010,05508 090 S\~{a}o Paulo, Brazil}
\email{marcosmalex@yahoo.de, malex@ime.usp.br}

\author{Marco Radeschi}

\address{ 
Mathematisches Institut\\
 WWU M\"unster, Einsteinstr. 62, M\"unster, Germany.}

\email{mrade\_02@uni-muenster.de}

\thanks{The first author was supported by a 
research productivity scholarship  from CNPq-Brazil
 and partially supported by FAPESP (S\~{a}o Paulo, Brazil). The second author was partially supported by Benjamin Franklin Fellowship at the University of Pennsylvania}

\subjclass{Primary 53C12, Secondary 57R30}

\keywords{Singular Riemannian foliations, Myers-Steenrod theorem, orbit spaces}

\begin{abstract}
In this paper we prove that an  isometry between  orbit spaces of two proper isometric actions is smooth if it preserves  the codimension of the orbits or if the orbit spaces have no boundary. In other words,  we  generalize   Myers-Steenrod's theorem for orbit spaces. These results are proved  in the more general context of singular Riemannian foliations.
\end{abstract}


\maketitle

%
%
\section{Introduction}

Given a Riemannian manifold $M$ on which a compact Lie group $G$ acts by isometries, the quotient $M/G$ is in general not a manifold. Nevertheless, the canonical projection $\pi:M\to M/G$ gives $M/G$ the structure of a Hausdorff metric space. Moreover, following  Schwarz \cite{Schwarz} one can define a ``smooth structure'' on $M/G$ to be the $\mathbb{R}$-algebra $C^{\infty}(M/G)$ consisting of functions $f:M/G\to \mathbb{R}$ whose pullback $\pi^*f$ is a smooth, $G$-invariant function on  $M$. If $M/G$ is a manifold, the smooth structure defined here corresponds to the more familiar notion of smooth structure.
A map $F:M/G\to M'/G'$ is called \emph{smooth} if the pull-back of a smooth function $f\in C^{\infty}(M'/G')$ is a smooth function in $M/G$. 

These concepts can actually be formulated in the wider context of singular Riemannian foliations. 
A singular foliation $\F$ on $M$ is called \emph{singular Riemannian foliation} (SRF for short) if 
 every geodesic  perpendicular  to one leaf   is perpendicular to every leaf it meets, see \cite[page 189]{Molino}.

A typical example of a singular Riemannian foliation is  the partition of a Riemannian manifold into the connected components of the orbits of an isometric action. Such singular Riemannian foliations are called \emph{Riemannian homogeneous}.

Given $(M,\F)$, one can define a quotient $M/\F$, also called \emph{leaf space}. If the leaves of $\fol$ are closed, $M/\fol$ can again be endowed with a metric structure and a smooth structure, exactly as in the case of group actions.

When dealing with Riemannian manifolds, a theorem of Myers and Steenrod states that the metric structure of a Riemannian manifold uniquely determines its smooth structure. In the same way, one can ask whether the metric structure on a quotient $M/G$ or $M/\F$ uniquely determines its smooth structure in the sense described above. This question can be restated in the following way: given an isometry
\[
F: M/\F\to M'/\F'
\]
between the quotients of two Riemannian manifolds, is $F$ smooth?

Classic theorems, like the Chevalley Restriction Theorem \cite{Chevalley} and the Luna-Richardson Theorem \cite{LunaRichardson} give a positive answer when $\F,\F'$ come from  some special group actions. Recently, Alexander Lytchak and the first named author generalized the results above, answering the question in the positive for special foliations $\F,\F'$ (namely \emph{infinitesimally polar} foliations, cf. \cite{AlexLytchak}). Nevertheless, a general answer to this question is not known, even for  isometric group actions.

In the present paper we provide a new sufficient condition for an isometry to be smooth.

\begin{theorem}
\label{main-theorem}
Let $M_1$ and $M_2$ be complete Riemannian manifolds  and $(M_{1},\F_{1})$, $(M_{2},\F_{2})$ be singular Riemannian foliations with closed leaves. 
Assume that there exists an isometry $\varphi: M_{1}/\F_{1}\to M_{2}/\F_{2}$ that preserves the codimension of the leaves. Then $\varphi$ is a smooth map. 
\end{theorem}

\begin{remark}
Notice that not every isometry $\varphi:M_1/\F_1\to M_2/\F_2$, which preserves the codimension of the leaves, lifts to a foliated diffeomorphism  $M_1\to M_2$. This fact can be illustrated with examples constructed
via a procedure called \emph{suspension of homomorphism}, see e.g. \cite[sec. 3.7]{Molino}.
Also notice that in \cite{FKM} the authors produce arbitrary numbers of pairwise non isometric foliations $(V_i,\F_i)$ on  vector spaces of the same dimension, having isometric 2-dimensional leaf spaces and the same codimension of the leaves. 
\end{remark}

\begin{remark}
The above theorem implies that if $M_{i}/\F_{i}$ are  isometric orbifolds, then they are diffeomorphic in the sense of Schwarz and hence in the classical sense, see e.g., Strub \cite{Strub} and  Swartz \cite[Lemma 1]{Swartz}. 
\end{remark}

 In the special case of leaf spaces without boundary (see definition in Section \ref{preliminaries}), a small modification in the proof of Theorem \ref{main-theorem} allow us to prove the next result.

\begin{theorem}
\label{proposition-main-theorem-no-boundary}
Let $(M_i,\F_i)$, $i=1,2$, be singular Riemannian foliations with closed leaves, and $\varphi:M_1/\F_1\to M_2/\F_2$ be an isometry. If $M_1/\F_1$ has no boundary, then $\varphi$ is smooth.
\end{theorem}

As an immediate corollary of Theorem \ref{main-theorem}, we obtain the following
\begin{corollary}\label{flow-is-smooth}
Let $(M,\F)$ be a  singular Riemannian foliation with closed leaves and $\varphi: M/\F\times (-\epsilon,\epsilon)\to M/\F$ a continuous family of isometries $\varphi_t:M/\F\to M/\F$ such that $\varphi_0=\textrm{id}_{M/\F}$. Then each $\varphi_t$ is smooth.
\end{corollary}

\begin{remark}
Flows of isometries on the leaf spaces of  foliations appear naturally in the study of the dynamical behavior of \emph{non closed} singular Riemannian foliations. Recall that a (locally closed) singular Riemannian foliation $(M,\F)$ is locally described by  submetries $\pi_{\alpha}:U_{\alpha}\to U_{\alpha}/\F_{\alpha}$, where $\{U_{\alpha}\}$ is an open cover of $M$ and $\F_{\alpha}$ denotes the restriction of $\F$ to $U_{\alpha}$. If a leaf $L$ is not closed, one might be interested to understand how it intersects a given neighborhood $U_{\alpha}$, and in particular how the closure $\overline{L}$ of $L$ intersects $U_{\alpha}$. In the regular case, the local quotient $U_{\alpha}/\F_{\alpha}$ is a manifold, and it turns out (cf. \cite[Thm 5.2]{Molino}) that the projection $\pi_{\alpha}(\overline{L}\cap U_{\alpha})$ is a submanifold, which is spanned by flows of isometries $\varphi_{\alpha}$ on $U_{\alpha}/\F_{\alpha}$. As one tries to generalize this result to singular Riemannian foliations, the main difficulty is that the local quotient $U_{\alpha}/\F_{\alpha}$ is no longer  a manifold. In particular, when studying the smoothness of the flows of isometries $\varphi_{\alpha}$ (which still exist) one cannot rely on classical theorems anymore, hence the need to develop new techniques to deal with these  more general situations. Corollary \ref{flow-is-smooth} is a first result in this direction. Other results on this topic are the center of a forthcoming paper. 
\end{remark}

The paper is organized as follows: in Section \ref{sketch} we give an outline of the proof of Theorem \ref{main-theorem}, when $\F_1$ and $\F_2$ come from an isometric group actions. In Section \ref{preliminaries}  we recall  some definitions and results that we will use later on. Finally, in Section \ref{proof1} we prove Theorem \ref{main-theorem} and Theorem \ref{proposition-main-theorem-no-boundary}.



\dank The authors are grateful to Alexander Lytchak for inspiring the main questions of this work, and for very helpful discussions and suggestions. The authors also thank Wolfgang Ziller, Dirk T\"{o}ben, Ricardo Mendes and Renato Bettiol for useful suggestions.
%
%

\bigskip

\section{Sketch of the proof of Theorem \ref{main-theorem}}
\label{sketch}

We provide here the main ideas of the proof of Theorem \ref{main-theorem} in the special case of singular Riemannian foliations $(M_i,\F_i)$, $i=1,2$, given by the orbits of a closed group $G_i$ acting by isometries of $M_i$ (\emph{Riemannian homogeneous foliation}). For more details on group actions, we refer the reader to \cite[Ch. 5]{PalaisTerng}. As in Theorem \ref{main-theorem} we assume that there exists an isometry $\varphi:M_1/G_1\to M_2/G_2$ that preserves the codimension of the orbits. 

We divide the proof into 3 steps:
\\

\emph{Step 1: Reducing the problem to Euclidean space.}

Let $p_i\in M_i$, $L_i=G_i(p_i)$ the orbit through $p_i$, $V_i^{\perp}=\nu_pL_i$ the normal space to the orbit, and $S_{p_i}=\exp_{p_i}(V_i^{\perp}) \cap B_{\epsilon}(p_i)$ a \emph{slice} through $p_i$. The isotropy group $G_{p_i}=({G_i})_{p_i}$ acts on $V_i^{\perp}$ by the \emph{isotropy representation}, and by the Slice Theorem the exponential map $\exp_{p_i}:V_i^{\perp}\cap B_{\epsilon}(0)\to S_{p_i}$ descends to a smooth map
\[
\exp_{i}^*:V_i^{\perp}/G_{p_i}\to S_{p_i}/(\F_i\cap S_{p_i})\simeq U_i/G_i
\]
where $U_i=B_{\epsilon}(L_i)$ and $U_i/G_i$ is a small neighborhood of the projection $p_i^*$ of $p_i$, in $M_i/G_i$.

Suppose that $\varphi(p_1^*)=p_2^*$. We first observe that the restriction $\varphi: U_1/G_1\to U_2/G_2$ gives rise to an isometry $\varphi_*=(\exp_2^*)^{-1}\circ \varphi \circ \exp_1^*:V_1^{\perp}/G_{p_1}\to V_2^{\perp}/G_{p_2}$ between quotients of Euclidean spaces $V_1^{\perp}$, $V_2^{\perp}$. In fact, it is known that the Euclidean metric $\metric_{p_i}$ on $V_i^{\perp}$ is the limit of a family of local metrics $\metric^i_{\lambda}$ constructed via ``homothetic transformations'' around $p_i$. It easy to check that $\varphi:(U_1/G_1, \metric^1_{\lambda} ) \to (U_2/G_2,  \metric^2_{\lambda} )$ is an isometry, and by taking the limit as $\lambda\to 0$ we obtain that $\varphi_*$ is an isometry. In order to prove the smoothness of $\varphi$, it is enough to prove the smoothness of $\varphi$ around $p_1^*$. In particular, it is enough to prove that $\varphi_*$ is smooth, and this concludes Step 1.

\emph{Step 2: Exploiting the properties of Euclidean space.}

Since $V_i^{\perp}$ is an Euclidean space, there is an explicit relation between the eigenvalues of the shape operator of a principal orbit in a direction $x$, and the zeroes of special Jacobi fields, called \emph{projectable Jacobi fields}, along the geodesic with initial velocity $x$. Since the zeroes of projectable Jacobi fields can be read off from the quotient $V_i^{\perp}/G_{p_i}$ and from the dimension of the orbits, it follows that the eigenvalues of the shape operator remain constant along a principal orbit. In particular, the mean curvature vector field $H_i$ of the principal orbits projects to a well defined vector field $H_i^*$ in the principal part of $V_i^{\perp}/G_{p_i}$.

For the same reason, if $\varphi$ preserves the codimension of the orbits, then the focal points of an orbit $q_1^*\in V_1^{\perp}/G_{p_1}$ correspond to those of the orbit $q_2^*=\varphi(q_1^*)\in V_2^{\perp}/G_{p_2}$. In particular we obtain the following
\begin{lemma}
\label{sketch-lemma1}
   $  d \varphi (H_{1}^*)=H_{2}^*$ in the quotient of the principal stratum.
\end{lemma}

\emph{Step 3: Proving the smoothness of $\varphi$.}

The goal here is to prove that for any smooth $G_{p_2}$-invariant function $f$ on $V_2^{\perp}$, the pullback $\varphi^*f$ is a smooth $G_{p_1}$-invariant function in $V_1^{\perp}$. 
First of all, by using Taylor expansion of first order and basic properties of isometries, we prove that the gradient of $\varphi^*f$, which is well defined on each stratum, is also continuous when passing from one stratum to another. This implies that $\varphi$ is of class $C^1$. 

Secondly, we apply Lemma \ref{sketch-lemma1} above to prove that, on the regular parts, we have $\varphi^*\bigtriangleup_{V_2^{\perp}}f=\bigtriangleup_{V_1^{\perp}} \varphi^{*}f$. This follows from the fact that the Laplacian $\bigtriangleup_{V_i^{\perp}}$ can be obtained from $\bigtriangleup_{V_i^{\perp}}/G_{p_i}$ by 
\begin{equation}
\label{eq-sketch-lemma3}
 \bigtriangleup_{V_i^{\perp}}f= \bigtriangleup_{V_i^{\perp}/G_{p_i}}f -\metric_{i}(\nabla f , H^*_{i}) 
\end{equation}
and $\varphi$ preserves both summands in the right hand side.

Finally, we prove that $\varphi^*\bigtriangleup_{V_2^{\perp}}=\bigtriangleup_{V_1^{\perp}} \varphi^{*}$ holds weakly everywhere.
\begin{lemma}
\label{sketch-lemma3}
$\bigtriangleup \varphi^{*} f=\varphi^{*} \bigtriangleup f$ holds weakly for any smooth $G_{p_2}$-invariant function $f$ in $V_2^\perp$.
\end{lemma}
The above lemma and the fact that $\varphi$ is of class $C^1$ 
 allow us to exploit the regularization properties of the Laplacian to prove the smoothness of $\varphi^*f$ (and therefore that of $\varphi$) via a bootstrap type argument.

\begin{remark} In \cite{Terng} the author uses similar techniques to prove a generalization of Chevalley Restriction Theorem. In that situation, the author has to prove that certain functions are polynomial, and this is proved by showing that every derivative of high enough order is zero. The approach in the present paper uses bootstrap lemma instead.
\end{remark}

%
%

\bigskip

\section{Preliminaries}
\label{preliminaries}

\begin{center}

\end{center}
\subsection*{The leaf space}
Let $(M,\F)$ be a singular Riemannian foliation with closed leaves. The foliation induces an equivalence relation $\sim$ on $M$, where $p\sim q$ if and only if $p,q$ lie in the leame leaf. The quotient $M/\sim$ is called \emph{leaf space} of $(M,\F)$ and is denoted by $M/\F$. The canonical map $\pi: M\to M/\F$ gives $M/\F$ with structure of a Hausdorff metric space, where the distance between two point is given by the distance of the corresponding leaves.
Also recall that the image of a \emph{stratum} $\Sigma$ (the set of leaves with the same dimension) is an orbifold of dimension $\dim \pi(\Sigma)=\dim \Sigma -\dim \F|_{\Sigma}$. If $M_{reg}$ denotes the \emph{regular stratum} (the set of leaves with maximal dimension), the \emph{quotient codimension} of $\Sigma$ is
\[
qcodim(\Sigma)=\dim\pi(M_{reg})-\dim \pi(\Sigma)=\dim M-\dim \F-\dim \Sigma +\dim \F|_{\Sigma}.
\]
To say that $M/\F$ \emph{has no boundary} is equivalent to requiring that $qcodim(\Sigma)>1$ for every singular stratum.

The metric space $M/\F$ has a natural smooth structure. More precise, one can define the ring $C^{\infty}(M/\F)$ of \emph{smooth functions} on $M/\F$ to be the ring of functions $f:M/\F\to \RR$ whose pullback $\pi^*f$ is a smooth function on $M$. Notice that by contruction $\pi^*f$ is \emph{basic}, i.e., it is constant along the leaves of $\F$.

A map $F:M_1/\F_1\to M_2/\F_2$ is said to be \emph{smooth}  if for every smooth function $f\in C^{\infty}(M_2/\F_2)$ (in the sense defined above) the pullback $F^*f$ is again a smooth function $C^{\infty}(M_1/\F_1)$. By definition, the canonical projection $\pi:M\to M/\F$ is smooth and a submetry. Moreover, when restricted to the regular part $M_{reg}\to 
M_{reg}/\F$ it is a Riemannian submersion.

Given a point $p\in M$ or a vector $x\in T_pM$, we will denote with $p^*, x^*$ the projections $\pi( p)$, $\pi_*(x)$ respectively.

\subsection*{Non connected foliations. }Along this paper, we will have to consider Riemannian foliations with non connected leaves. This kind of foliations comes up naturally: consider for example a Riemannian homogeneous foliation $(M,G)$. Even if $G$ itself is connected, some isotropy subgroup might not be, and the orbits of $G_p$ under the slice representation might also be disconnected. Therefore the Riemannian homogeneous foliation $(\nu_pM,G_p)$ would be an example of a disconnected singular Riemannian foliation. In general, a \emph{singular Riemannian foliation with disconnected leaves} $(M,\F)$ is a triple $(M,\F^0,\mathsf{K})$ where $(M,\F^0)$ is a (usual) SRF, $\mathsf{K}$ is a group of isometries of $M/\F^0$, and the \emph{non-connected leaves} of $\F$ are just the orbits $\mathsf{K}\cdot L_p$, for $L_p\in \F^0$.

A leaf $L$ of a disconnected foliation $\F$ is called a \emph{principal leaf} if it satisfies the following conditions:
\begin{enumerate}
\item[(1)] each connected component of $L$ is a \emph{principal leaf} of $\F^0$, i.e., a \emph{regular leaf} (a leaf with maximal dimension) that has a trivial holonomy;  see e.g \cite[page 22]{Molino}.
\item[(2)] If there exists an isometry $k\in \mathsf{K}$ which fixes any component of $L$ in $M/\F^0$, $k$ is the identity.
\end{enumerate}

\subsection*{Infinitesimal foliation. }Let $(M,\F)$ be a singular Riemannian foliation with closed leaves. Given a point $p\in M$, let $V_p^{\perp}=\nu_pL_p$, and for some $\epsilon>0$, let $S_p=\exp_p (V_p^{\perp}) \cap B_{\epsilon}(p)$ be a \emph{slice} through $p$, where $B_{\epsilon}(p)$ is the distance ball of radius $\epsilon$ around $p$. In the definition of $S_p$, we assume $\epsilon$ to be small enough that $S_p$ does not contain any focal point of $L_p$.  The foliation $\F$ induces a foliation $\F|_{S_p}^0$ on $S_p$ by letting the leaves of $\F|_{S_p}^0$ be the connected components of the intersection between $S_p$ and the leaves of $\F$. In general,  the foliation $(S_p,\F|_{S_p}^0)$ is not a singular Riemannian foliation with respect to the induced metric on $S_p$. Nevertheless, the \emph{pull-back} foliation $\exp_p^{*}(\F^0)$ is a singular Riemannian foliation on $V_p^{\perp}\cap B_{\epsilon}(0)$ equipped with the Euclidean metric (cf.~\cite[Proposition 6.5]{Molino}), and it is invariant under homotheties fixing the origin (cf.~\cite[Lemma 6.2]{Molino}). In particular, it is possible to extend $\exp^*(\F^0)$ to all of $V_p^{\perp}$, giving rise to a singular Riemannian foliation $(V_p^{\perp},\F^0_p)$ called the \emph{infinitesimal foliation} of $\F$ at $p$. The fundamental group $\pi_1(L_p)$ acts on $V_p^{\perp}/\F^0_p$ by \emph{holonomy maps} in such a way that it induces a disconnected foliation $(V_p^{\perp},\F_p)=(M,\F_p^0,\pi_1(L_p))$. Via the exponential map, the leaves of $\F_p$ correspond to the intersections of the leaves of $\F$ with $S_p$ (i.e., we no longer restrict to the connected components), and in particular the exponential map $\exp_p: V_p^{\perp}\cap B_{\epsilon}(0)\to S_p\In M$ defines a diffeomorphism $\exp_*$ between $(V_p^{\perp}\cap B_{\epsilon}(0))/\F_p$ and a neighborhood of $p^*=\pi(p)$ in $M/\F$.

If $(M,G)$ is Riemannian homogeneous foliation, the infinitesimal foliation $(V_p^{\perp},\F^0_p)$ (respectively the disconnected infinitesimal foliation $(V_p^{\perp},\F_p)$) is again Riemannian homogeneous foliation, given by the action of the identity component of the isotropy group $G_p^0$ (respectively the action of the whole isotropy group $G_p$) on $V_p^{\perp}$.

\subsection*{Orbifold part of the leaf space} Let $(M,\F)$ be a singular Riemannian foliation with closed leaves. A point $p^*\in M/\F$ is called \emph{orbifold point} of $M/\F$ if there is a neighborhood of $p^*$ isometric to a quotient $U/\Gamma$, where $U$ is a Riemannian manifold and $\Gamma$ is a finite group of isometries. The set of orbifold points of $M/\F$ is denoted by $(M/\F)_{orb}$ and called the \emph{orbifold part} of $M/\F$. By \cite{LytchakThorbergsson2} the preimage of $(M/\F)_{orb}$ consists of those points whose infinitesimal foliation is polar, and the complement of $(M/\F)_{orb}$ in $M/\F$ has codimension $\geq 2$.

%
%

\bigskip

\section{Proof of Theorems \ref{main-theorem} and \ref{proposition-main-theorem-no-boundary}}
\label{proof1}

Suppose we have two closed singular Riemannian foliations $(M_i,\F_i)$, $i=1,2$, and an isometry $\varphi:M_1/\F_1\to M_2/\F_2$ that preserves the codimension of the orbits. For $p_i\in M_i$, denote $p_i^*$ its projection under the canonical map $\pi_i:M_i\to M_i/\F_i$.

In order to avoid cumbersome notations, we will denote each basic function $f:M_i\to \mathbb{R}$ and the induced function on $M_i/\F_i$ by the same letter $f$.


We now prove Theorem \ref{main-theorem}, closely following the steps presented in Section \ref{sketch}. We first observe that the main problem can be reduced to a problem in Euclidean space, following standard arguments from the theory of SRF's; see \cite{Molino,AlexToeben2,LytchakThorbergsson2}. Chosen $p_1,p_2$ so that $\varphi(p_1^*)=p_2^*$, $\varphi$ restricts to an isometry $\varphi: (S_{p_1},\metric_1)/\F_1 \to (S_{p_2},\metric_2)/\F_2$.
 Recall that the flat metrics $\metric_{p_i}$  are the limit of metrics $\metric_{\lambda}^{i}=\frac{1}{\lambda^{2}} h_{\lambda}^{*}\metric_{i}$ as $\lambda\to0$, where $h_{\lambda}$ denotes the homothetic transformation around $p_i$. In particular, since the isometry $\varphi$ induces an isometry $\varphi:  (S_{p_1},\metric_{\lambda}^1)/\F_1 \to (S_{p_2},\metric_{\lambda}^2)/\F_2$ for any $\lambda\in(0,1)$, by taking the limit as $\lambda\to0$ we obtain an isometry
 \[
 \varphi_*: (V_{p_1}^{\perp},\metric_{p_1})/\F_{p_1} \to (V_{p_2}^{\perp},\metric_{p_2})/\F_{p_2}.
 \]
This is an isometry between leaf spaces of foliations in Euclidean space. Moreover, around $p_1$,
 $\varphi$ can be written as $\exp^*_2\circ \varphi_*\circ (\exp_1^*)^{-1}$, where $\exp_i^*$ are diffeomorphisms, and therefore $\varphi$ is smooth around $p_1$ if and only if $\varphi_*$ is smooth. Thus in order to prove the theorem, it is enough to check it on Euclidean spaces.
\\

\begin{proposition}
\label{proposition-isometry-preserve-meancurvature}
Let $(\RR^{n_1},\fol_1)$, $(\RR^{n_2},\fol_2)$ be two (possibly non-connected) SRF's with closed leaves, and let $\varphi:\RR^{n_1}/\fol_1\ra \RR^{n_2}/\fol_2$ be an isometry that preserves the codimension of the leaves. Then the mean curvature vector fields  of the corresponding principal leaves are basic and $\varphi$ preserves the projections  of those vector fields.
\end{proposition}
\begin{proof}
This result was proved in Gromoll and Walschap \cite[Theorem 4.1.1]{GromollWalschap} in the case of regular Riemannian foliations. 
 In what follows we will explain how that proof can be adapted in the case of SRF's. 

For $i=1,2$ let $p_i\in M_{i}=\RR^{n_i}$ be a principal point of $\F_i$ such that $\varphi^*(p_1^*)=p_2^*$. Moreover, let $x_i\in V_{p_i}^{\perp}$, $i=1,2$ be horizontal vectors such that $\varphi_*(x_1^*)=x^*_2$. Finally, define $\gamma_i(t)=p_i+tx_i$.

In order to prove the proposition, it is enough to show that $\tr(\Sh_{x_1})=\tr(\Sh_{x_2})$, where $\Sh_{x_i}$ is the \emph{shape operator}  of the leaf $L_{p_i}$ through $p_i$. We will actually show something stronger, namely that every nonzero eigenvalue of $\Sh_{x_1}$ is an eigenvalue of $\Sh_{x_2}$ of the same multiplicity, for almost every $x_1$.

 Since the complement of the orbifold part $(M_1/\F_1)_{orb}$ has codimension $\geq 2$, almost every projected horizontal geodesic stays in  $(M_1/\F_1)_{orb}$ for all time, and  in what follows we will assume that our fixed geodesic $\gamma_1$ has this property.

 Because $\varphi\left((M_1/\F_1)_{orb}\right)=(M_2/\F_2)_{orb}$, and $\varphi$ takes projected  horizontal geodesics in $(M_1/\F_1)_{orb}$ 
to projected horizontal geodesics in $(M_2/\F_2)_{orb}$, we conclude that $\varphi(\pi_1\circ \gamma_1)=\pi_2\circ \gamma_2$; see \cite{Swartz}.

On the one hand, since $\pi\circ\gamma_i$ are contained in $(M_i/\F_i)_{orb}$, we know that the $\varphi$ preserves conjugate points along
 $\pi\circ\gamma_i$, as well as their multiplicity. We also know, by hypothesis, that $\varphi$ preserves codimension of the singular points contained in $\gamma_1$. 

On the other hand, by \cite[Lemma 5.2]{LytchakThorbergsson2} the focal index, i.e., the number of focal points of $L_{p_i}$ along $\gamma_i$ counted with multiplicity, is a sum of two indices, namely:
\begin{itemize}
\item the horizontal index, which counts conjugate points of $\pi(p_i)$ with their multiplicity along $\pi\circ\gamma_i$. The notion of conjugare point along $\pi\circ\gamma_i$ makes sense, since $\pi\circ\gamma_i$ is contained in the orbifold part of $M/\F$.
\item The vertical index, which counts the singular points of $\F$ contained in $\gamma_i$, their multiplicity being the jump in codimension $\textrm{codim}\, L_{\gamma_i(t)}-\textrm{codim}\, \F$ at those points; see also discussion in \cite[Section 5.2]{LytchakThorbergsson2}. 
\end{itemize}
These facts combined, imply that $\varphi$ preserves the focal points of $L_{p_{i}}$ along $\gamma_i$ and their multiplicities. 
                          
Finally recall that, since $M_i$ are Euclidean spaces, 
the focal points of $L_{p_1}$ along $\gamma_1$ are at distance $1/\lambda_1,\ldots 1/\lambda_r$, where
$\{\lambda_1,\ldots, \lambda_r\}$ are the eigenvalues of $\Sh_{x_1}$ counted with the same multiplicity, see \cite[Proposition.4.1.8]{PalaisTerng}. Since $\varphi$ preserves focal points and their multiplicities, we infer that the 
shape operator $\Sh_{x_2}$ of the leaf $L_{p_2}$ has the same eigenvalues as those of $\Sh_{x_1}$, counted with the same multiplicity. In particular, $\tr(\Sh_{x_1})=\tr(\Sh_{x_2})$ whenever the projection of $\gamma_1(t)=p_1+tx_1$ is entirely contained in $(M_1/\F_1)_{orb}$.
Because this condition is open and dense, the fact that $\varphi$ preserves mean curvature vector field  follows from the continuity of the mean curvature form.

\end{proof}


\begin{remark}
\begin{enumerate}
\item The above proposition implies that, given a SRF $\F$ on $\mathbb{R}^{n}$, then each principal leaf $L$ of $\F$ is a \emph{generalized isoparametric} submanifold, i.e., the principal curvatures along a basic vector field of $L$ are constant. 
\item By the proof of Proposition \ref{proposition-isometry-preserve-meancurvature} above, it follows that if $M_1/\F_1$, $i=1,2$, has no boundary, then almost every horizontal geodesic $\gamma_1$ stays in the regular stratum of $M_1$, and all the focal points of $L_{p_1}$ along $\gamma_1$ correspond to conjugate points of $\pi_1(p_1)$ along $\pi_1\circ \gamma_1$. In particular, if $M_1/\F_1$ has no boundary, $\varphi$ preserves the mean curvature even without the assumption of preserving the codimension of the leaves.
 \end{enumerate}
\end{remark}

By the discussion above, Theorems \ref{main-theorem} and \ref{proposition-main-theorem-no-boundary} will both be proved once we show that any isometry between leaf spaces preserving the (basic) mean curvature vector fields is smooth. In order to do this, we show:

\begin{proposition}
\label{proposition-C1}
Let $M_1$ and $M_2$ be complete Riemannian manifolds and $(M_{1},\F_{1})$ and $(M_{2},\F_{2})$ be SRF's with closed leaves. Then an isometry $\varphi: M_{1}/\F_{1}\to M_{2}/\F_{2}$ is of class $C^1$, i.e., for each smooth basic function $f$ on $M_2$, the basic function $\varphi^{*}f$ on $M_1$ is of class $C^{1}$.
\end{proposition} 
\begin{proof}
For $i=1,2$ let $p_i$ be a point in $M_i$, let $P_i$ be a small tubular neighborhood of $L_{p_i}$ in the stratum containing $p_i$, and let $U_i$ be a small tubular neighborhood of $P_i$ of radius $\epsilon$, with closest-point projection $\P_i:U_i\to P_i$. We can make these choices so that $\varphi(\pi_1(p_1))=\pi_2(p_2)$ and $\varphi(\pi_1(P_1))=\pi_2(P_2)$.

If $f$ is a smooth basic function on $M_2$, let $f_{0}$ be the smooth basic function on $U_2$ defined as $f_{0}=\P_2^*(f|_{P_2})$.
Since the gradient of $f$ in $p_2$ is tangent to the stratum, $\nabla f_0= \nabla f$ at $p_2$. Therefore, if we rewrite $f$ as $f=f_0+R$ (locally this is the Taylor formula), we conclude that $\nabla R= 0$ at $p_2$.

The pullback of $f_{0}$ under $\varphi$ is
\[
\varphi^*f_{0}=\varphi^*(\P_2^*(f|_{P_2}))=\P_1^*\big((\varphi^*f)|_{P_1}\big).
\]
It is easy to prove that $\varphi^*f$ is smooth on each stratum of $M_1$, in particular $(\varphi^*f)|_{P_1}$ is smooth and thus $\varphi^*f_{0}$ is smooth on $U_1$. If we write
\[
\varphi^*f=\varphi^*f_0+\varphi^*R,
\]
it now follows that $\varphi^*R$ is smooth on each stratum, and it makes sense to define the gradient $\nabla \varphi^*R$ on each stratum. Moreover, since $\varphi^*R$ is basic, $\nabla \varphi^*R$ is always horizontal and we can compute $\lim_{p\to p_1} \|\nabla \varphi^*R\|$ from the quotient:
\begin{equation}\label{to-zero}
\lim_{p\to p_1} \|\nabla \varphi^*R\|( p)=\lim_{p''\to \varphi(\pi_1( p))} \|\nabla R\|(p'')=0,
\end{equation}
where we used the fact that $\varphi$ is an isometry.

Equation \eqref{to-zero} implies that $\varphi^*R$ is of class $C^1$ ar $p_1$ and $\nabla \varphi^*R(p_1)=0$. In particular $\varphi^*f=\varphi^*f_0+\varphi^*R$ is $C^1$ at $p_1$, and this proves the proposition.
\end{proof}

\begin{remark}
\begin{enumerate}
\item In Proposition \ref{proposition-C1}, the fact that $\varphi$ is an isometry is used only in equation \eqref{to-zero}. Here, all we have really used is the fact that 
 the derivative of $\varphi$ (restricted to each stratum) is locally bounded.
\item Observe that Proposition \ref{proposition-C1} does not use the assumption that $\varphi$ preserves the codimension of the leaves. In particular, every isometry between leaf spaces is of class $C^1$. 
\end{enumerate}
\end{remark}

The next proposition concludes the proof of Theorem \ref{main-theorem} and Theorem \ref{proposition-main-theorem-no-boundary}.
\begin{proposition}
\label{theorem-myers-steenrod}
Let $M_1$ and $M_2$ be complete Riemannian manifolds and $(M_{i},\F_{i})$ SRF's with closed leaves 
such that the mean curvature vector fields $H_{i}$ of the corresponding principal leaves are basic. 
Assume that there exists an isometry $\varphi: M_{1}/\F_{1}\to M_{2}/\F_{2}$ that preserves the mean curvature vector fields  restricted to the principal stratum. Then $\varphi$ is a smooth map. 
\end{proposition}
\begin{proof}
Let $\metric_i$ denote the metric on $M_i$.
Recall that we are using the notation ${H_i}_*$ to denote the projection ${\pi_i}_*H_i$ of the mean curvature vector field on the regular part of $M_i$.
For $i=1,2$, let $p_i$ be a regular point in $M_i$, and let $U_i$ be a neighborhood of $p_i$ that admits a local quotient $\Q_i:U_i\to B_i$, where the manifold $B_i$ is the local model of the orbifold $\pi_i(U_i)\In M_i/\F_i$. We can make these choices so that $\varphi(\pi_1(p_1))=\pi_2(p_2)$ and $\varphi(\pi_1(U_1))=\pi_2(U_2)$. Since
\[
\varphi|_{\pi_1(U_1)}: \pi_1(U_1)\to \pi_2(U_2)
\]
 is an isometry, by \cite{Swartz} it lifts to an isometry $\ol{\varphi}: B_1\to B_2$.

Let $f$ be a smooth basic function of $(M_{2},\F_{2})$. We want to prove that $\varphi^{*}f$ is a smooth basic function of $(M_{1},\F_{1})$.

Clearly $f$ stays basic with respect to $\F_2|_{U_2}$ and thus it defines a function on $B_2$, which we still denote $f$.

We recall that (see e.g., \cite[page 53]{GromollWalschap})
\begin{equation}
\label{eq-0-theorem-myers-steenrod}
\bigtriangleup_{U_i}f= \bigtriangleup_{B_i} f -\metric_{i}( \nabla f, {H_{i}}_*). 
\end{equation}

Set $u:=\bigtriangleup_{M_2} f.$ Equation \eqref{eq-0-theorem-myers-steenrod} implies that  $u$ is a smooth basic function on $(U_{2},\F_{2})$. 

Since $\ol{\varphi}: B_1\to B_2$ is an isometry and $\varphi_* {H_{1}}_*={H_{2}}_*$ by assumption, it easily follows from equation \eqref{eq-0-theorem-myers-steenrod} that
\begin{equation}
\label{eq-1-theorem-myers-steenrod}
\bigtriangleup_{M_1} (\varphi^{*} f)=\varphi^{*} u \qquad \textrm{on }U_1.
\end{equation}
Since $p_1$, $p_2$ were chosen arbitrarily, it follows that $\bigtriangleup (\varphi^{*} f)=\varphi^{*} u $ in the regular part $(M_1)_{reg}$. Since the complement of $(M_1)_{reg}$ in $M_1$ is a locally finite union of submanifolds of codimension $\geq 2$, equation \eqref{eq-1-theorem-myers-steenrod} holds weakly on the whole $M_1$ by the following Lemma.

\begin{lemma}\label{lemma1-theorem-myers-steenrod}
Let $f,u$ be $C^1$ functions on a manifold $M$, and let $M'$ be a submanifold of $M$ such that:
\begin{itemize}
\item $M\setminus M'$ is a locally finite union of submanifold of codimension $\geq 2$.
\item $f,u$ are smooth on $M'$, and $\bigtriangleup f=u$ on $M'$.
\end{itemize}
Then $\bigtriangleup f=u$ holds weakly on $M$, i.e.,

$$\int_{M} f \cdot \bigtriangleup h=\int_{M} u \cdot h$$
for every smooth function  $h$  with compact support on $M$.
\end{lemma} 
\begin{proof}
Let $W$ be a neighborhood of $M\setminus M'$ with smooth boundary that $\partial W$. Let $h$ be a smoth function with compact support on $M$.
By Green's second identity
\begin{equation}\label{eq-1b-theorem-myers-steenrod}
\int_{M-W} \bigtriangleup f \cdot h -\int_{M-W}f \cdot\bigtriangleup h = \int_{\partial W} h\cdot \metric( \nabla f,\eta) -f \cdot \metric( \nabla h,\eta),
\end{equation}
where $\eta$ is the normal vector field of $\partial W$.

Since $\bigtriangleup f=u$ on $M-W\In M'$, equation \eqref{eq-1b-theorem-myers-steenrod} becomes
\begin{equation}
\label{eq-2-theorem-myers-steenrod}
\int_{M-W} u\cdot h -\int_{M-W} f\cdot \bigtriangleup h = \int_{\partial W} h\cdot \metric( \nabla f,\eta) -f\cdot\metric( \nabla h,\eta).
\end{equation}
Fixed an arbitrary positive $\epsilon$, it is possible to choose a small neighborhood $W$ so that
\begin{equation}
\label{eq-3-theorem-myers-steenrod}
\left| \int_{M} u\cdot h - \int_{M-W} u\cdot h \right| < \frac{\epsilon}{3},
\end{equation}
\begin{equation}
\label{eq-4-theorem-myers-steenrod}
\left| \int_{M} f\cdot \bigtriangleup h - \int_{M-W} f\cdot \bigtriangleup h \right| < \frac{\epsilon}{3}.
\end{equation}
Since $M\setminus M'$ has codimension $\geq 2$, we can choose $W$ with boundary of arbitrarily small volume. In particular we can assume
\begin{equation}\label{eq-5-theorem-myers-steenrod}
\left|\int_{\partial W}h\cdot \metric( \nabla f,\eta) -f\cdot\metric( \nabla h,\eta)\right|\leq |\partial W|\cdot \sup_{M} | h\cdot \metric( \nabla f,\eta) -f\cdot \metric( \nabla h,\eta) | < \frac{\epsilon}{3}.
\end{equation}
Equations \eqref{eq-2-theorem-myers-steenrod} through \eqref{eq-5-theorem-myers-steenrod} now prove the Lemma.
\end{proof}

By Lemma \ref{lemma1-theorem-myers-steenrod} above, the equation $\bigtriangleup \varphi^{*} f=\varphi^{*} u$ holds weakly on the whole $M_1$. Since $\varphi^{*}u$ is a function of class $C^{1}$ (recall Proposition \ref{proposition-C1}) we can apply the regularity theory of solutions of linear elliptic equations (see e.g., the  proof of Theorem 3, Section 6.3.1 of Evans \cite{Evans}), and this proves the smoothness of $f$. Therefore $\varphi$ is smooth as well, and Proposition \ref{theorem-myers-steenrod} follows.

\end{proof}


\bibliographystyle{amsplain}

\end{document}